\documentclass[11pt]{article}

\usepackage{amsmath,amssymb,bm,bbm,mathrsfs,amscd}
\usepackage{color}
\usepackage{amsthm}
\usepackage{dsfont}
\usepackage{graphicx}
\usepackage{epsfig}
\usepackage{subfigure}

\usepackage[a4paper,top=3cm,bottom=3cm,left=3.5cm,right=3.5cm]{geometry}

\def\qed{\hfill \vrule height 7pt width 7pt depth 0pt\medskip}
\def\beq{\begin{equation}}
\def\eeq{\end{equation}}

\newtheorem{theorem}{Theorem}

\newtheorem{proposition}[theorem]{Proposition}
\newtheorem{lemma}[theorem]{Lemma}
\newtheorem{corollary}[theorem]{Corollary}
\newtheorem{example}{Example}
\theoremstyle{remark}

\newcommand{\ba}{\begin{array}}
\newcommand{\ea}{\end{array}}

\newcommand{\be}{\begin{equation}}
\newcommand{\ee}{\end{equation}}

\newcommand{\mc}{\mathcal}

\newcommand{\1}{\mathbbm{1}}

\newcommand{\N}{\mathbb{N}}

\renewcommand{\P}{\mathbb{P}}

\DeclareMathOperator*{\argmax}{argmax}

\def\1{\mathds{1}}

\def\N{\mathbb{N}}

\def\P{\mathbb{P}}

\providecommand{\keywords}[1]{\textbf{Key words.} #1}
\providecommand{\ams}[1]{\textbf{AMS subject classification.} #1}

\newcommand{\TheTitle}{A game theoretic approach to a network cloud storage problem}

\title{\TheTitle}

\author{
  Fabio Fagnani\thanks{Department of Mathematical Sciences, Politecnico di Torino, Italy
    (fabio.fagnani@polito.it).}
  \and
  Barbara Franci\thanks{Department of Mathematical Sciences, Politecnico di Torino, Italy
   (barbara.franci@polito.it).}
}

\begin{document}
\maketitle

\begin{abstract}
The use of game theory in the design and control of large scale networked systems is becoming increasingly more important. In this paper, we follow this approach to efficiently solve a network allocation problem  
motivated by peer-to-peer cloud storage models as alternatives to classical centralized cloud storage services. To this aim, we propose an allocation algorithm that allows the units to use their neighbors to store a back up of their data. We prove convergence, characterize the final allocation, and corroborate our analysis with extensive numerical simulation that shows the good performance of the algorithm in terms of scalability, complexity and structure of the solution.
\end{abstract}


\ams{91A22, 91A43, 90C27} 

\section{Introduction}

In the last decade, game theory has emerged as a new fundamental paradigma in the design and control of large scale networked engineering systems \cite{GTW, GTL, GTDS}. The basic idea is that of modeling system units as rational agents whose behavior consists in a selfish search for their maximum utility. The goal is to design both the agents utility functions and an adaptation rule (evolutionary game dynamics) in such a way that the resulting global behavior is desirable. One of the advantages of this approach with respect to classical global optimization techniques is that it naturally leads to scalable decentralized solutions that can easily incorporate constraints on the computational power of the units as well as the information flow determined by the network architecture.  Consensus dynamics, which is a fundamental ingredient of networked control, has been recently recognized \cite{CCPG} to be an instance of an evolutionary dynamics of a networked coordination game.

Two are the challenging issues when (evolutionary) game theory is used as a mechanism design: first, designing admissible agent utility functions such that the resulting game possesses desirable Nash equilibria; second, designing adaptation rules that guarantee the solution's convergence to such Nash equilibria.  For an important class of games known as potential games, a popular adaptation rule is the so-called noisy best-response where units choose their state by following a Gibbs-Boltzmann distribution having its peak on the maxima of the utility function \cite{GTL}. This rule can be proven to yield convergence, asymptotically in time and in the limit of the noise approaching zero, to maxima of the potential that form a subclass of the Nash equilibria.  In this context, the goal is to design utility functions in such a way that the potential (in particular, its maxima) have the desired global properties. The noisy best-response turns out to be a random decentralized scheme for the maximization of the potential.

In this paper, we consider an allocation problem where a population of units is connected through a network and each of them possesses a number of items that need to be allocated among the neighboring units. The interpretation is that of a decentralized peer-to-peer cloud storage model where computers (or other devices with storage capabilities) in a network are made to collaborate with each other so that each of them can store a back up of their data in the available space of the neighbors. The model was first presented in \cite{GTA} where the reader can find more detailed motivations. 

We cast the allocation problem into a game theoretic framework by assuming that each unit has a utility function that gives a value to their neighbors on the basis of their reliability (possibly reflecting responsiveness or security issues), their current congestion (resources have bounded storage capabilities), and the amount of data the unit already stores in them. This reflects a number of features that we want the storage final configuration to possess. In the storage process, a unit should prefer to use those neighbors that are more reliable and, among equally reliable ones, those that are less congested. Finally, storage should not be split into too many parts to avoid excessive communication burden. Following evolutionary game theory \cite{GTL}, we propose an algorithm based on a noisy best response action: each time a unit activates, it decides which neighbor to use for storage on the basis of a Gibbs probability distribution.

Allocation games have been extensively studied \cite{ACG, CRAG}. One of the key features of our model is the presence of hard constraints as a consequence of the bounded storage capabilities of the units. This property is non-classical in game theory and has remarkable consequences on the structure of Nash equilibria and the behavior of the algorithm. The proposed game is shown to be a potential game. However, convergence of the best response dynamics is here a subtle issue that does not follow from classical results. This will follow instead from a careful proof of the ergodicity of the underlying Markov process that constitutes the technical core of this paper. The conclusive result is that: (i) all units complete their allocation in finite time with probability one; and (ii) the allocation configuration converges, when the noise parameter approaches $0$, to a maximum of the potential. These results were announced in \cite{GTA}, but not proven. Moreover, in this paper, we deal with a more general situation: units can be in off mode and hence not responsive to storage requests from neighbors. 

In the remaining part of this section, we formally define the storage allocation problem and recall some basic facts proven in \cite{GTA} (in particular, a necessary and sufficient condition for the allocation problem to be solvable). Section \ref{sec: algorithm} is devoted to cast the problem to a potential game theoretic framework \cite{ROS} and to propose a distributed algorithm that is an instance of a noisy best response dynamics. The main technical part of the paper is Section \ref{sec: analysis} where the fundamental results Theorem \ref{theo main 1} and Corollary \ref{cor main 2} are stated and proven. Theorem \ref{theo main 1} ensures that the algorithm  reaches a complete allocation with probability one, if a complete allocation is indeed possible. Corollary \ref{cor main 2} studies the asymptotic behavior of the algorithm and explicitly exhibits the invariant probability distribution. Consequence of Corollary \ref{cor main 2} is that in the double limit when time goes to infinity and the noise parameter goes to $0$, the algorithm converges to a Nash equilibrium that is, in particular, a global maximum of the potential function. This guarantees that the solution will indeed be close to the global welfare of the community. Finally, Section \ref{sec simulations} is devoted to the presentation of an extensive set of simulations that corroborate the theoretical results, prove the good scalability properties of the algorithm in terms of speed and complexity, and illustrate the effect of the parameters of the utility functions in the solution reached by the algorithm. A conclusion section ends the paper.

\subsection{The model}

Consider a set $\mc X$ of units that play the double role of users who have to allocate externally a back up of their data, as well resources where data from other units can be allocated. Generically, an element of $\mc X$ will be called a unit, while the terms user and resource will be used when the unit is considered in the two possible roles of, respectively, a source or a recipient of data. We assume units to be connected through a directed graph $\mc G=(\mc X,\mc E)$ where a link $(x,y)\in\mc E$ means that unit $x$ is allowed to store data in unit $y$. We denote by $$N_x:=\{y\in\mc X\,|\, (x,y)\in\mc E\},\quad N^-_y:=\{x\in\mc X\,|\, (x,y)\in\mc E\}$$
respectively, the out- and the in-neighborhood of a node. Note the important different interpretation in our context: $N_x$ represents the set of resources available to unit $x$ while $N^-_y$ is the set of units having access to resource $y$. If $D\subseteq \mc X$, we put $N(D)=\cup_{x\in D}N_x$ and $N^-(D)=\cup_{x\in D}N_x^-$.

We imagine the data possessed by the units to be quantized atoms of the same size. Each unit $x$ is characterized by two non negative integers:
\begin{itemize}
\item $\alpha_x$ is the number of data atoms that unit $x$ needs to back up into his neighbors, 
\item $\beta_x$ is the number of data atoms that unit $x$ can accept and store from his neighbors.
\end{itemize}
The numbers $\{\alpha_x\}$ and $\{\beta_x\}$ will be assembled into two vectors denoted, respectively, $\alpha$ and $\beta$. We also define 
$$\mc A_x=\{(x,a)\,|\, a\in\{1,\dots ,\alpha_x\}\},\quad \mc A=\bigcup_{x\in\mc X}\mc A_x$$
Given the triple $(\mc G, \alpha, \beta)$, we define an allocation 
as any map $Q:\mc A\to\mc X$  ($Q(x,a)=y$ means that $x$ has allocated the data atom $a$ into resource $y$) satisfying the properties expressed below.
\begin{enumerate}
\item[(C1)] {\bf Graph constraint} $Q(x,a)\in N_x$ for all $x\in\mc X$ and $a\in \{1,\dots ,\alpha_x\}$;
\item[(C2)] {\bf Storage limitation} For every $y\in\mc X$, 
it holds, $ |Q^{-1}(y)|\leq \beta_y$.
\end{enumerate}

We will say that the allocation problem is solvable if an allocation $Q$ exists. We denote by $\mc Q$ the set of allocations. We will also need to consider partial allocations, namely functions $Q : D\to\mc X$ where $D\subseteq \mc A$ and satisfying, where defined, conditions (C1) and (C2). We denote by $\mc Q_p$ the set of partial allocations.

The following result gives a necessary and sufficient condition for the existence of allocations. The proof, which follows from Hall's theorem, can be found in \cite{GTA}.

\begin{theorem}\label{theo allocation exists} Given $(\mc G, \alpha, \beta)$, there exists an allocation iff the following condition is satisfied:
\beq\label{cond allocation exists}
\sum\limits_{x\in D}\alpha_x\leq \sum\limits_{y\in N(D)}\beta_y\quad\forall D\subseteq \mc X
\eeq
\end{theorem}

In general, it is not necessary to check the validity of (\ref{cond allocation exists}) for every subset $D$. We say that $D\subseteq \mc X$ is maximal if for any $D'\supsetneq D$, it holds $N(D')\supsetneq N(D)$. We say that $D_1, D_2\subseteq \mc X$ are independent if $N(D_1)\cap N(D_2)=\emptyset$ and $D\subseteq \mc X$ is called irreducible if it can not be decomposed into the union of two non empty independent subsets. Clearly, it is sufficient to verify (\ref{cond allocation exists}) for the subclass of maximal irreducible subsets.

\begin{example} If $\mc G$ is complete, we have that $N(\{x\})=\mc X\setminus\{x\}$ while $N(D)=\mc X$ for all $D$ such that $|D|\geq 2$. Hence, the only maximal irreducible subsets are the singletons $\{x\}$ and the set $\mc X$. Condition (\ref{cond allocation exists}) in this case reduces to
\beq\label{cond allocation exists complete}
\alpha_x\leq \sum\limits_{y\neq x}\beta_y,\; \forall x\in\mc X\qquad 
\sum\limits_{x\in\mc X}\alpha_x\leq \sum\limits_{y\in \mc X}\beta_y
\eeq
\end{example}
In general, the class of maximal irreducible subsets can be large and grow more than linearly in the size of $\mc X$, as the following example shows.
\begin{example} If $\mc G=(\mc X, \mc E)$ is a line graph ($\mc X=\{1,2,\dots , n\}$ and $\mc E=\{(i, i+1), \,i=1, \dots , n-1\}$) it can be checked that the maximal irreducible subsets are those of the form $\{i, i+2, \dots , i+2s\}$. 
\end{example}

When the allocation problem is solvable, it will in general admit many possible solutions. In order to analyze the set of solutions it is convenient to introduce an aggregative function of allocations.
Given a partial allocation $Q$, consider the matrix $W(Q)\in\N^{\mc X\times \mc X}$ where $W(Q)_{xy}$ is the number of atomic data that $x$ has allocated in $y$ under $Q$, namely,
\begin{equation}\label{W}W(Q)_{xy}:=|Q^{-1}(y)\cap\mc A_x|\end{equation}
Clearly $W=W(Q)$ satisfies the following conditions
\begin{enumerate}
\item[(P1)]  $W_{xy}\geq 0$ for all $x,y$ and $W_{xy}=0$ whenever $(x,y)\not\in \mc E$.
\item[(P2)]  $W_x:=\sum\limits_{y\in\mc X}W_{xy}\leq\alpha_x$ for all $x\in\mc X$.
\item[(P3)] $W^y:=\sum\limits_{x\in\mc X}W_{xy}\leq \beta_y$ for all $y\in\mc X$.
\end{enumerate}
It is immediate to see that, conversely, if there exists $W$ satisfying these properties (such a $W$ is called a partial allocation state), then, from it, we can construct an allocation $Q$ such that $W=W(Q)$. Clearly, under this correspondence, we have that $Q\in \mc Q$ iff $W$ satisfies (P2) with equality for all $x\in\mc X$. In this case $W$ is called an allocation state.
The set of partial allocation states and the set of allocation states are denoted, respectively, with the symbols $\mc W_p$ and $\mc W$. It is clear that two allocations $Q^1$ and $Q^2$ such that $W(Q^1)=W(Q^2)$, only differ for a permutation of the data atoms of the various units and for many purpouses can be considered  as equivalent. All the quantity of interest for the game theoretic setting will be defined at the level of $W$.

Allocation states are in general not unique:

\begin{example}\label{ex: multiple solutions}
Suppose $\mc G$ is the complete graph and that  
$\alpha_x=\alpha$, $\beta_x=\beta$ for every  $x\in\mc X$, with $\alpha\leq\beta$.  Among the possible allocation states there are those where each unit uses only one resource, namely for any permutation $\sigma :\mc X\to\mc X$ without fixed points, we can consider
\beq\label{extremal Nash}W^\sigma_{xy}=\alpha\1_{\sigma(x)=y}\eeq
More generally, if $s$ is an integer dividing $\alpha$ and we consider any symmetric $s$-regular graph over $\mc X$ (without self-loops) with adjacency matrix $A$, we can consider the corresponding state allocation $W^A$ given by
\beq\label{regular Nash}W^A_{xy}=\frac{\alpha}{s}A_{xy}\eeq
It is immediate to check that this is indeed a state allocation and that, in particular, $(W^A)^y=\alpha$ for every $y$. 
\end{example}

In general, there are extra desired features that we may want solutions to possess:

\begin{itemize}
\item {\bf (Reliability)} resources will often have different reliability properties and, in the allocation, a preference to more reliable resources should be given;
\item {\bf (Congestion)} equally reliable resources should be equally used, avoiding congestion phenomena;
\item {\bf (Aggregation)} units should use as few resources as possible to allocate their data.
\end{itemize}
Notice how in the example above, all the allocations states considered yield equal congestion for all resources while they differ from the point of view of aggregation: the solutions $W^\sigma$ are of course the one maximizing this feature.

Instead of translating the above qualitative features into quantitative hard constraints on the allocation states, in this paper, we undertake a different road. We cast the problem into a game theoretic scenario defining, for each unit $x$ and for each resource $y$, a function $f_{xy}(W)$ that is the utility of $x$ in using $y$ under the current (possibly partial) allocation state $W$. The utility functions are defined following the so called 'mechanism design' approach: utility will tend to be larger for more reliable or less congest resources or for resources where the unit has already allocated a larger size of its data. This will force the Nash equilibria of the game to be allocation states possessing, at a certain degree, the above features (typically in a compromised way depending on the parameter tuning of the utility functions). Our algorithm builds upon this game theoretic set-up. It is based on evolutionary game theory and is fully distributed: units iteratively allocate and move their data among the neighbors on the basis of the space available and trying to maximize a utility function that incorporates the features considered above. The algorithm always leads to an allocation, when such an allocation exists, that is a Nash equilibrium of the game.  


\section{The game theoretic set-up and the algorithm}\label{sec: algorithm}
In this section we present the details of the game theoretic structure and of the allocation algorithm.

Utilities are defined as follows. Under a (possibly partial) allocation state $W$, the utility of a unit $x$ in using resource $y$ is given by
\beq\label{payoff units}
f_{xy}(W):=\lambda_y- k_cW_y/\beta_y+k_aW_{xy} 
\eeq
The term $\lambda_y$ encodes possible reliability differences among resources, the second term is a congestion term that takes into consideration the level of use of the resource, and, finally, the third term depends on both $x$ and $y$ and pushes a unit to allocate in those resources where it has already allocated. $k_c, k_a$ are two non-negative parameters to tune the effect of the congestion and of the aggregation terms, respectively. 

Given $W\in\mc W_p$ and $x\in\mc X$, put 
$$\mc X^x(W):=\{y\in N_x\;|\; W_y<\beta_y\}$$
the set of resources still available for $x$ under the current allocation state $W$. 
An allocation state $W\in\mc W$ (and also any $Q\in\mc Q$ such that $W(Q)=W$) is called a Nash equilibrium if, for every $x\in\mc X$, for every $y\in N_x$ such that $W_{xy}>0$, for every $y'\in \mc X^x(W)$, it holds
$$f_{xy}(W)\geq f_{xy'}(W')$$
where $W'=W-e_{xy}+e_{xy'}$.
Under a Nash equilibrium, a unit whose goal is to maximize its utility, has no advantage in moving their allocated data, under the standing assumption that only one data atom at a time can be moved. Notice that data atoms are to be interpreted as aggregations of data and the decision of their size is part of the design of the algorithm. Clearly different levels of granularity will give rise to different game models including different Nash equilibria. 

Existence of Nash equilibria follow from the following considerations. Define
\beq\label{potential discrete W} \Psi(W):=\sum\limits_{y\in\mc X}\sum\limits_{s=0}^{W_y}[\lambda_y- k_cs/\beta_y]+k_a\sum\limits_{x,y\in\mc X}\sum\limits_{s=0}^{W_{xy}}s\eeq
and notice that if $W, W'\in\mc W$ are such that $W'=W-e_{x\bar y}+e_{x\bar y'}$, then,
\beq\label{potential relation}\begin{array}{l}\Psi(W')-\Psi(W)\\


=(\lambda_{\bar y'}- k_cW'_{\bar y'}/\beta_{\bar y'}+k_aW'_{\bar x\bar y'})-(\lambda_{\bar y}- k_cW_{\bar y}/\beta_{\bar y}+k_aW_{\bar x\bar y})\\
=f_{x\bar y'}(W')-f_{x\bar y}(W)\end{array}\eeq
In other terms, under the allocation state $W$, when user $x$ moves a data atom from $\bar y$ to $\bar y'$, he experiences a variation in utility given by $\Psi(W')-\Psi(W)$. $\Psi$ is called a potential of the game. 
%
Maxima of $\Psi$ are clearly Nash equilibria while, in general, the converse is not true. Considering that $\Psi$ is defined on a finite set, a maximum, and thus a Nash equilibrium, always exists. The existence of a potential will be also fundamental in exploiting the properties of the allocation algorithm.

\medskip
\noindent
{\bf Remark (Relation with classical game theory):} 
The game theoretic structure we have just defined is apparently different from the classical one: a set of players, each player playing an option and incurring in a utility that depends on the option he plays as well from the options played by the rest of the population. What are the players, the options and the utilities on our framework? One way to recast our model into a classical game theoretic one is that of considering as set of players not the set of units $\mc V$, but the set of data atoms to be stored $\mc A$. Given $(x,a)\in\mc A$, its set of options is simply given by $N_x$, the resources where it can be stored.
A choice of an option for each player is equivalent to an allocation $Q$. The utility of a player $(x,a)$ under an allocation $Q$ is given by $f_{xy}(W(Q))$. The game so defined is clearly potential with potential function given by $\Psi$.

\medskip
A non-classical feature of our game is the presence of hard constraints ($W_y\leq\beta_y$) in the way options can be simultaneously played. The presence of such constraints have important consequences in the structure of Nash equilibria. 

\subsection{Examples}
We now present a number of simple examples where Nash equilibria can be explicitly exhibited. The simplest case is when $k_a=0$ since in this case the game becomes a pure congestion game.

\begin{example}\label{ex: homogeneous1}
Suppose $\mc G$ is the complete graph and that  
$\alpha_x=\alpha$, $ \beta_x=\beta$,  $\lambda_x=\lambda$ for every $x\in\mc X$
with $\alpha\leq\beta$. Suppose moreover that $k_a=0$.
It is immediate to check that any $W$ such that $W^y=\alpha$ for every resource $y\in\mc X$ is a Nash equilibrium.  We now check that these are all the possible ones. Indeed if $W$ is a Nash equilibrium such that $W^y>\alpha >W^{y'}$ for some $y, y'\in\mc X$, it follows that 
$$f_{xy}(W)=\frac{W^y}{\beta}<\frac{W^{y'}+1}{\beta}=f_{xy'}(W+e_{xy'}-e_{xy})$$
This yields $W_{xy}=0$ for every $x\neq y'$. Therefore, $W_{y'y}=W^y>\alpha$ that is a contradiction.
Among the Nash equilibria considered above, there are those introduced in Example \ref{ex: multiple solutions}, $W^{\sigma}$ and $W^A$.
Since the potential $\Psi$ depends on $W$ only through the values $\{W^y\}$, it is easy to check that these Nash points are all maxima for $\Psi$.
\end{example}

In the case when $\mc G$ is not complete, the existence of state allocations $W$ where all resources are equally used (e.g.$W^y=\alpha$ for all $y$)  is not longer guaranteed. From Theorem \ref{theo allocation exists}, a necessary and sufficient condition for this to happen is clearly $|D|\leq |N(D)|$ for every $D\subseteq \mc X$. Notice that when this condition is verified and we consider the bipartite graph $\tilde{\mc G}=(\mc X\times \mc X, \tilde{\mc E})$ where $(x,y)\in\tilde{\mc E}$ iff $(x,y)\in E$, Hall's theorem guarantees the existence of a matching in $\tilde {\mc G}$ complete on the first set, namely a permutation $\sigma:\mc X\to\mc X$ such that $(x,\sigma(x))\in\tilde{\mc E}$ for every $x\in\mc X$. Using $\sigma$, as in the complete case, we can construct the state allocation $W^{\sigma}$ where each unit uses just one resource. Note that the condition $|D|\leq |N(D)|$ is automatically verified in symmetric $d$-regular graphs. Indeed, if $\mc E_{D}$ is the set of edges having one of the nodes in $D$, we have that $d|D|=|\mc E_{D}|\geq d|N(D)|$. 

We now present an example where the condition $|D|\leq |N(D)|$ is violated.

\begin{example} Let $\mc X=\{0,1,2\}^2$ and $\mc G$ to be the $2$-grid over $\mc X$, Suppose $\alpha_x=\alpha$,  $\beta_x=\beta$, $\lambda_x=\lambda$ for every $x\in\mc X$
with $\alpha\leq\beta$. Suppose moreover that $k_a=0$. Notice that $|D|\leq |N(D)|$ is not always verified. Indeed, if $D=\{(i,j)\,|\; 2|(i+j)\}$ we have that $|D|=5$ while $|N(D)|=4$. This implies that allocations where all the resources are equally used can not exist. More precisely, it can not exist an allocation if $\beta <5\alpha /4$. For $\beta\geq 5\alpha /4$ an allocation exists:
$$W_{(ij)(h,k)}=\left\{\begin{array}{ll} \frac{\alpha}{2}\; &{\rm if} i,j\in\{0,2\}\\ [2pt]
\frac{\alpha}{4}\;&{\rm if}\, (i,j)=(1,1)\\ [2pt]
\frac{2\alpha}{5}\; &{\rm if}\, 2\not{|}i+j,\, h,k\in\{0,2\}\\ [2pt]
\frac{\alpha}{5}\; &{\rm if}\, 2\not{|}i+j,\, (h,k)=(1,1)
\end{array}\right.
$$
This yields:
$W^{(i,j)}=4\alpha/5$ if $(i,j)\in D$ while $W^{(i,j)}=5\alpha/4$ if $(i,j)\in \mc X\setminus D$. It is immediate to verify that the neighbors of each unit are resources equally congested: this implies that $W$ is a Nash equilibrium and, actually, for previous considerations, it is easy to show that is indeed a maximum of $\Psi$.
\end{example}

The case when $k_a>0$ is generally harder to analyze. We start with a simple remark. Suppose $y_1,y_2\in\mc X$ are two resources with the same affidability and storage capability $\beta$. Let $W$ be a Nash equilibrium. If a unit $x$ is using both resources under $W$ (e.g. $W_{xy_i}>0$ for i=1,2) and none of them is fully congested (e.g. $W^{y_i}<\beta$ for $i=1,2$), then the fact that $W$ is Nash implies the double inequality:
\beq\label{equilibrium}\frac{W^{y_1}-W^{y_2}}{\beta}+k_a-\frac{1}{\beta}\leq k_a(W_{xy_1}-W_{xy_2})\leq
\frac{W^{y_1}-W^{y_2}}{\beta}-k_a+\frac{1}{\beta}\eeq
This implies that $k_a\leq 1/\beta$. Therefore, when $k_a>1/\beta$, in a Nash equilibrium, a unit will never use more than one resource, unless all but one of them are fully congested. 

\begin{example}\label{ex: homogeneous2}
Suppose $\mc G$ is the complete graph and that  
$\alpha_x=\alpha$, $\beta_x=\beta$, $\lambda_x=\lambda$ for every $x\in\mc X$
with $\alpha\leq\beta$. It is easy to show that the state allocations $W$ defined in (\ref{extremal Nash}) are maxima of the potential also for any  $k_a>0$. In the case when $k_a\leq 1/\beta$, it follows from (\ref{equilibrium}) that also the state allocations $W^A$ proposed in (\ref{regular Nash}) are still Nash equilibria (no longer maxima though if $k_a>0$ and $s>1$). Instead, in the case $k_a>1/\beta$, it follows again from (\ref{equilibrium}) that $W^A$ is Nash only in the case $\beta=\alpha$, namely when resources are all fully congested. In the case $k_a>1/\beta$, moreover, new Nash equilibria show up where resources have different congestion level. As a specific example, for the case $|\mc X|=3$, the state allocation
$$W=\left(\begin{matrix}0 &0&\alpha\\ 0 &0&\alpha\\ \alpha &0&0\end{matrix}\right)$$
is a Nash equilibrium for $\beta\geq 2\alpha$ and $k_a\geq 1/\beta$.
\end{example}

Even when the graph is complete and $k_a=0$, the presence of resources with different affidability and the presence of hard storage constraints, may cause the presence of Nash equilibria where resources are not uniformly used. This is shown in the following example.

\begin{example}\label{ex: different lambda}
Suppose $\mc X=\{1,2,3\}$ and $\mc G$ is the complete graph (namely a $3$-cycle). Values of $\alpha$, $\beta$ are given by
\begin{center}
\begin{tabular}{c c }
	$\alpha_1=7$& $\beta_1=35$  \\
	$\alpha_2=35$&  $\beta_2=21$ \\
	$\alpha_3=14$ & $\beta_3=28$  \\
\end{tabular}
\end{center}
Suppose moreover that $$\lambda_1-1>\lambda_2=\lambda_3=0$$
so that unit $1$, as a resource, is always the best to use, even when it is fully congested.
Of course, in any Nash equilibrium $W$, resource $1$ will be saturated, therefore $W$ will have the following structure
$$W=\left(\begin{matrix}0 &a &7-a\\ b & 0&35-b\\ 35-b &b-21&0\end{matrix}\right)$$
with $a\in [0,7]\cap\N$ and $b\in [21,35]\cap\N$.
Now notice that under a state like $W$ above, units $2$ and $3$ are automatically at equilibrium, they have no convenience to move atoms. Let us analyze the utilities of unit $1$. Defined
$W'=W+e_{13}-e_{12}$ and $W''=W+e_{12}-e_{13}$, we have that
$$\begin{array}{ll}f_{12}(W)=\frac{21-a-b}{21},\;& f_{13}(W')=\frac{a+b-43}{28}\\[3pt]
f_{13}(W)=\frac{a+b-42}{28},\;& f_{13}(W'')=\frac{20-a-b}{21}
\end{array}
$$
Algebraic computations show that
$$\begin{array}{ll}f_{12}(W)\geq f_{13}(W')\, &\Leftrightarrow\, a+b\leq 30\\ 
f_{13}(W)\geq f_{12}(W'')\, &\Leftrightarrow\, a+b\geq 30
\end{array}
$$

We can thus conclude as follows: $W$ is a Nash equilibrium if and only if the parameters $a$ and $b$ are in the following three ranges:
$$\begin{array}{lll}
(I)\;\;&21\leq b<23 &a=7\\
(II)\;\; &23\leq b\leq 30 &a=30-b\\
(III)\;\; &30< b\leq  35 & a=0
\end{array}
$$
Notice that in case (II), we have that $a+b=30$: in this case resources $2$ and $3$ are equivalent in the sense that there is no convenience in moving data from one to the other and viceversa. Instead, in case (I) and (II), resources $2$ and $3$ give unbalanced utilities. Straightforward computation shows that, Nash equilibria of type II correspond to global maxima of the potential function $\psi$.


\end{example}

%

\subsection{The algorithm}

The allocation algorithm we are proposing is fully distributed and asynchronous and is only based on communications between units taking place along the links of the graph $\mc G=(\mc X, \mc E)$. 

At every instant of time, a unit is either in functional state on or off: we interpret off units, as units temporarily disconnected from the network, not available for communication and any action including storage and data retrieval. A unit, which is currently in state on, can activate and either allocate or move an atomic piece or its data among the available resources (e.g. those neighbors that still have place available and are on at that time). Choice of the available resource where to allocate will be following a classical stochastic relaxation of the utility maximization based on a Gibbs probability distribution: this is a classical choice in evolutionary game theory and will be akin to a fairly complete theoretical analysis.

We now describe all the mathematical details of the algorithm. We recall that a partial allocation is a function $Q:D\subseteq \mc A\to\mc X$ and that $W=W(Q)$ denotes the corresponding partial allocation state. The functional state of the network at a certain time will instead be denoted by $\xi\in\{0,1\}^\mc X$: $\xi_x=1$ means that the unit $x$ is on. As already said, the basic ingredient of the algorithm is the Gibbs distribution that we now formally introduce. Given a functional state $\xi\in\{0,1\}^{\mc X}$, a partial allocation state $W\in\mc W_p$ and a unit $x\in\mc X$, we define
$$\mc X^x(\xi, W):=\{y\in N_x\,|\, W_y<\beta_y,\, \xi_y=1\}$$
that is the set of neighbors of $x$ that are on and have still space available.
We then denote with $p^x(\xi, W)$ the Gibbs probability distribution over $\mc X^x(\xi, W)$ given by
$$p^x_y(\xi, W)=\frac{e^{\gamma f_{xy}(W+e_{xy})}}{Z^x_{\gamma}(\xi, W)}\quad y\in \mc X^x(\xi, W)$$
where $\gamma>0$ is a parameter and where 
$$Z^x_{\gamma}(\xi, W)=\sum_{y\in\mc X^x(\xi, W)}e^{\gamma f_{xy}(W+e_{xy})}$$ is the normalizing factor.

The times when units modify their functional state (off to on or on to off) and the times when they activate to allocate/distribute are modeled as a family of independent Poisson clocks whose rates will be denoted (for unit $x$), respectively, $\nu_x^{on}$, $\nu_x^{off}$, and $\nu_x^{act}$. The functional state of the network as a function of time $\xi(t)$ is thus a continuous time Markov process whose components are independent Bernoulli processes. 

Mathematically, the algorithm can be described as a continuous-time process $Q(t):D(t)\to \mc X$ taking values on the set of partial allocations $\mc Q_p$. We let $W(t)=W(Q(t))$ denote the corresponding process of partial allocation states. 

The details of the algorithm are described below. 

\begin{enumerate}
\item Assume unit $\bar x$ activates at time $t$. With probabilities 
$$P^{\rm all}_{\bar x}(W(t)),\quad  P^{\rm dis}_{\bar x}(W(t))=1-P^{\rm all}_{\bar x}(W(t))$$ the resource $\bar x$ will make, respectively, an allocation or a distribution move as explained below. Of course we are assuming that 
$$\begin{array}{l}
P^{\rm all}_{\bar x}(W(t))=0 \; {\rm if}\; W(t)^{\bar x}=\alpha_{\bar x}\\[2pt]
P^{\rm dis}_{\bar x}(W(t))=0 \; {\rm if}\; W(t)^{\bar x}=0\end{array}$$

\item (ALLOCATION MOVE)
\begin{itemize}
\item Choose $(\bar x,\bar a)$ uniformly at random in $(D(t))^c\cap\mc A_{\bar x}$
\item  Choose $\bar y'$ according to the Gibbs probability $p^{\bar x}_{\bar y'}(\xi(t),W(t))$
\item Put $D({t^+})=D(t)\cup\{(\bar x,\bar a)\},$
$$\quad Q({t^+})(x,a)=\left\{\begin{array}{ll} \bar y'\;\; &{\rm if}\,(x,a)=(\bar x,\bar a)\\
Q(t)(x,a)\;\; &{\rm if}\, (x,a)\in D(t)\end{array}\right.$$
\end{itemize}
\item (DISTRIBUTION MOVE)
\begin{itemize}
\item Choose $(\bar x,\bar a)$ uniformly at random in $D(t)\cap\mc A_{\bar x}$ and put $\bar y=Q(t)(\bar x,\bar a)$
\item Choose $\bar y'$ according to the Gibbs probability $p^{\bar x}_{\bar y'}(\xi(t), W(t)-e_{\bar x\bar y})$
\item Put $D({t^+})=D(t),$
$$\quad  Q({t^+})(x,a)=\left\{\begin{array}{ll} \bar y'\;\; &{\rm if}\,(x,a)=(\bar x,\bar a)\\
Q(t)(x,a)\;\; &{\rm if}\, (x,a)\in D(t)\setminus \{(\bar x,\bar a)\}\end{array}\right.$$
\end{itemize}
\end{enumerate}

\section{Analysis of the algorithm}\label{sec: analysis}
In this section we analyze the behavior of the algorithm introduced above. We will essentially show two results. First, we prove that if the set of allocation $\mc Q$ is not empty (i.e. condition (\ref{cond allocation exists}) is satisfied), the algorithm above will reach such an allocation in bounded time with probability $1$ (e.g $D(t)=\mc A$ for $t$ sufficienty large). Second, we will show that, under a slightly stronger assumption than (\ref{cond allocation exists}), in the double limit $t\to +\infty$ and then $\gamma\to +\infty$, the process $Q(t)$ induced by the algorithm will always converge, in law, to a Nash equilibrium that is a global maximum of the potential function $\Psi$. 

In order to prove such results, it will be necessary to go through a number of intermediate technical steps. First of all, it will be convenient to work directly with the process $W(t)=W(Q(t))$. By the fact that the  jump probabilities of the original process $Q(t)$ are only function of $W(t)$ and of $\xi(t)$, it immediately follows that the augmented process $(\xi(t), W(t))$ is Markovian and its only non zero transition rates are described below:
\beq\label{transition1}
\Lambda_{(\xi, W), (\xi', W)}=
\left\{
\begin{array}{ll}
\nu_{\bar x}^{on}\quad &{\rm if}\, \xi_{\bar x}=0,\,\xi'_{\bar x}=1,\, \xi_x=\xi'_x\,\forall x\neq \bar x\\[2pt]
\nu_{\bar x}^{off}\quad &{\rm if}\, \xi_{\bar x}=1,\,\xi'_{\bar x}=0,\, \xi_x=\xi'_x\,\forall x\neq \bar x\end{array}
\right.
\eeq

\beq\label{transition2}
\Lambda_{(\xi, W), (\xi, W')}=
\left\{
\begin{array}{ll}
\nu_{\bar x}^{act}P^{\rm all}_{\bar x}(W)p_{\bar y'}^{\bar x}(\xi, W) \quad &\\
\hspace{3cm} {\rm if}\, \xi_{\bar x}=\xi_{\bar y'}=1,\, W'=W+e_{\bar x\bar y'}\\[2pt]
\nu_{\bar x}^{act}P^{\rm dis}_{\bar x}(W)\frac{W_{\bar x\bar y}}{\alpha_{\bar x}}p_{\bar y'}^{\bar x}(\xi, W-e_{\bar x\bar y})  \quad &\\
\hspace{3cm} {\rm if}\, \xi_{\bar x}=\xi_{\bar y}=\xi_{\bar y'}=1,\, W'=W+e_{\bar x\bar y'}-e_{\bar x\bar y}\end{array}
\right.
\eeq

We now introduce a graph on $\mc W_p$ that will be denoted by $\mc L_p$: an edge $(W,W')$ is present in $\mc L_p$ if and only if $\Lambda_{(\1,W),(\1, W')}>0$.

Our strategy will be to show that from any element $W\in\mc W_p$ there is a path in $\mc L_p$ to some element $W'\in \mc W$. Considering that there is a non-zero probability that the process $(\xi(t), W(t))$  starting from $(\xi, W)$ will reach $(\1, W)$ in bounded time (because of the presence of transitions (\ref{transition1})), standard probabilistic arguments will then allow to conclude that allocation (e.g. $D(t)=\mc A$) will be achieved in bounded time with probability $1$. 

Given $W\in \mc W_p$ we define the following subsets of units 
$$\mc X^f(W):=\{x\in\mc X\;|\; W^x=\alpha_x\},$$
$$\mc X^{sat}(W):=\{x\in\mc X\setminus\mc X^f(W)\;|\; \not\exists y\in N_x\;\hbox{\rm s.t}\; W(Q)_y< \beta_y\}$$
Units in $\mc X^f(W)$ are called fully allocated: these units have completed the allocation of their data under the state $W$. Units in $\mc X^{sat}(W)$ are called saturated: they have not yet completed their allocation, however, under the current state $W$, they can not make any action, neither allocate, nor distribute. 
Finally, define
$$\mc W_p^{sat}:=\{W\in\mc W_p\setminus\mc W\;|\; \mc X=\mc X^f(W)\cup \mc X^{sat}(W)\}$$
It is clear that from any $W\in\mc W_p\setminus \mc W_p^{sat}$, some allocation move can be performed. Instead, if we are in a state $W\in \mc W_p^{sat}$, only possibly fully allocated units can make a distribution move. Notice that, because of condition (\ref{cond allocation exists}), for sure there exist resources $y$ such that $W_y<\beta_y$ and these resources are indeed exclusively connected to fully allocated units. The key point is to show that in a finite number of distribution moves it is always possible to move some data atoms from resources connected to saturated units to resources with available space: this will then make possible a further allocation move.

For any fixed $W\in W_p$, we can consider the following graph structure on $\mc X$ thought as set of resources: $\mc H_W=(\mc X,\mc E_W)$. Given $y_1,y_2\in\mc X$, there is an edge from $y_1$ to $y_2$  if and only if there exists $x\in\mc X$ for which
$$W_{xy_1}>0,\quad (x,y_2)\in\mc E
$$
The edge from $y_1$ to $y_2$ will be indicated with the symbol $y_1\to_x y_2$ (to also recall the unit $x$ involved). The presence of the edge means that the two resources $y_1$ and $y_2$ are in the neighborhood of a common unit $x$ that is using $y_1$ under $W$. This indicates that $x$ can in principle move some of its data currently stored in $y_1$ into resource $y_2$ if this last one is available. We have the following technical result

\begin{lemma}\label{lemma equilibrium 1} Suppose $(\mc G ,\alpha, \beta)$  satisfies (\ref{cond allocation exists}). Fix $W\in W_p$ and let $\bar y\in\mc X$  be such that there exists $\bar x\in N_{\bar y}$ with $W^{\bar x}<\alpha_{\bar x}$.
Then, there exists a sequence
\beq\label{sequence} \bar y=y_0,\,x_0,\,y_1,\dots ,y_{t-1},\,x_{t-1},\,y_t\eeq
satisfying the following conditions
\begin{enumerate}
\item[(Sa)] Both families of the $y_k$'s  and of the $x_k$'s are each made of distinct elements;
\item[(Sb)] $y_k\to_{x_k} y_{k+1}$ for every $k=0,\dots ,t-1$;
\item[(Sc)] ${W}_{y_k}=\beta_{y_k}$ for every $k=0,\dots ,t-1$, and ${W(Q)}_{y_t}<\beta_{y_t}$.
\end{enumerate}
\end{lemma}
\begin{proof}
Let $\mc Y\subseteq \mc X$ be the subset of nodes that can be reached from $\bar y$ in $\mc H_W$.
Preliminarily, we prove that there exists $y'\in\mc Y$ such that $W_{y'}<\beta_{y'}$.
Let
$$\mc Z:=\{x\in\mc X\;|\; \exists y\in\mc Y,\, W_{xy}>0\}$$
and notice that, by the way $\mc Y$ and $\mc Z$ have been defined,
\beq\label{condZ} x\in\mc Z,\; (x,y)\in \mc E\;\Rightarrow\; y\in\mc Y\eeq
Suppose now that, contrarily to the thesis, $W_y\geq \beta_y$ for all $y\in\mc Y$. Then, 
\begin{equation}\label{estim}\sum\limits_{x\in \mc Z}\alpha_x \leq \sum\limits_{y\in \mc Y}\beta_y=\sum\limits_{y\in \mc Y}W_y=\sum\limits_{y\in \mc Y}\sum\limits_{x\in\mc Z}W_{xy}=\sum\limits_{x\in\mc Z}W^x<\sum\limits_{x\in \mc Z}\alpha_x\end{equation}
where the first inequality follows from (\ref{condZ}) and (\ref{cond allocation exists}), the first equality from the contradiction hypothesis, the second equality from the definition of $\mc Z$, the third equality again from (\ref{condZ}) and, finally, last inequality from the existence of $\bar x$. This is clearly absurd and thus proves our claim.

Consider now a path of minimal length from $\bar y$ to $\mc Y$ in $\mc H_W$:
$$\bar y=y_0\to_{x_0}y_1\to_{x_1}\dots \to_{x_{t-2}} y_{t-1}\to_{x_{t-1}}y_t$$
and notice that the sequence $\bar y=y_0,\,x_0,\,y_1,\dots ,y_{t-1},\,x_{t-1},\,y_t$ will automatically satisfy properties (Sa) to (Sc).
\qed
\end{proof}

We are now ready to prove the first main result.

\begin{theorem} \label{theo main 1} Assume that 
\begin{enumerate}
\item $\nu^{on}_x>0$ and $\nu^{act}_x>0$ for every $x\in\mc X$, 
\item $P_{x}^{\rm all}(W, \bar x)>0$ if $W^ {x}<\alpha_{x}$, 
\item $(\mc G ,\alpha, \beta)$  satisfies (\ref{cond allocation exists}). 
\end{enumerate}
Then, 
$$\P(\exists t_0\,|\, W(t)\in\mc W\,\forall t\geq t_0)=1$$
\end{theorem}
\begin{proof}
In order to prove the claim, it will be sufficient to show that from any $W\in\mc W_p$ there is a path in $\mc L_p$to some element $W'\in \mc W$. We will prove it by a double induction process. To this aim we consider two indices associated to any $W\in\mc W_p\setminus \mc W$. The first one is defined by
$$m_W=\sum\limits_{x\in\mc X}(\alpha_x-W^x)\geq 1$$
To define the second, consider any  $\bar x\in \mc X\setminus \mc X^f(W)$. We can apply Lemma \ref{lemma equilibrium 1} to $W$ and any $\bar y\in N_{\bar x}$ and obtain that we can find a sequence of agents 
$\bar y=y_0,\,x_0,\,y_1,\dots ,\,y_{t-1},\,x_{t-1},\,y_t$
satisfying the properties (Sa), (Sb), and (Sc) above. Among all the possible choices of $\bar x\in\mc X$, $\bar y\in N_{\bar x}$ and of the corresponding sequence, assume we have chosen the one minimizing $t$ and denote such minimal $t$ by $t_W$. The induction process will be performed with respect to the lexicographic order induced by the pair $(m_W, t_W)$. 

In the case when $t_W=0$, it means we can find $\bar x\in\mc X$, $\bar y\in\ N_{\bar x}$ such that $W_{\bar y}<\beta_{\bar y}$. Therefore, under the allocation state $W$, the unit $\bar x$ can allocate a further data atom to $\bar y$. Considering that the activation of the unit $\bar x$ has positive probability because of assumptions 1. and 2., this shows that $W$ is connected to a $W'$ such that $m_{W'}<m_W$. In case $m_W=1$, this means that $W'\in\mc W$. 

Consider now any $W\in\mc W_p\setminus \mc W$ such that $t=t_W>1$. Let $\bar x\in\mc X$, $\bar y\in\ N_{\bar x}$  and the sequence $\bar y=y_0,\,x_0,\,y_1,\dots ,\,y_{t-1},\,x_{t-1},\,y_t$
satisfying the properties (Sa), (Sb), and (Sc) above. In the allocation state $W$, the unit $x_{t-1}$, if activated (and this again has positive probability to happen because of assumptions 1. and 2.), can thus move an atomic piece of data from $y_{t-1}$ to $y_t$. The new allocation state is $W'=W-e_{x_{t-1}y_{t-1}}+e_{x_{t-1}y_t}$. Since $W'_{y_{t-1}}<\beta_{y_{t-1}}$, for sure $t_{W'}<t_W$. The induction argument is thus complete. \qed

\end{proof}

We are now left with studying the Markov process $W(t)$ on $\mc W$.  Noisy Gibbs best response dynamics are known to yields reversible Markov processes. This is indeed the case also in our case once the process has reached the set of allocations $\mc W$. Precisely, the following result holds:

\begin{proposition}\label{prop time rev} $(\xi(t), W(t))$, restricted to $\{0,1\}^\mc X\times \mc W$, is a time reversible Markov process. More precisely, for every $(\xi, W),( \xi',W')\in\{0,1\}^{\mc X}\times \mc W$  it holds
\beq\label{time reversible}\rho_{(\xi, W)}\Lambda_{(\xi, W),(\xi', W')}=\rho_{(\xi', W')}\Lambda_{(\xi', W'),(\xi, W)}\eeq
where
\beq \rho_{(\xi, W)}=\left[\prod\limits_{x:\,\xi_x=1}\nu_x^{\rm on}\prod\limits_{x:\,\xi_x=0}\nu_x^{\rm off} \right]
\displaystyle\frac{\prod_x\alpha_x!}{\prod_{x,y}W_{xy}!}
e^{\gamma \Psi(W)}\eeq
\end{proposition}
\begin{proof}
It follows from relations (\ref{transition1}) and (\ref{transition2}) that the only cases when $\Lambda_{(\xi, W),(\xi', W')}$ and $\Lambda_{(\xi', W'),(\xi, W)}$ are not both equal to zero are the following:
\begin{enumerate}
\item[(i)] $W'=W$, $\xi_{\bar x}=0,\,\xi'_{\bar x}=1,\, \xi_x=\xi'_x\,\forall x\neq \bar x$
\item[(ii)] $W'=W$, $\xi_{\bar x}=1,\,\xi'_{\bar x}=0,\, \xi_x=\xi'_x\,\forall x\neq \bar x$
\item[(iii)] $\xi'=\xi$, $W'=W+e_{\bar x\bar y'}-e_{\bar x\bar y}$.
\end{enumerate}
In case (i), we have that
$$\frac{ \rho_{(\xi, W)}}{ \rho_{(\xi', W')}}=\frac{\prod\limits_{x:\,\xi_x=1}\nu_x^{\rm on}\prod\limits_{x:\,\xi_x=0}\nu_x^{\rm off}}{\prod\limits_{x:\,\xi'_x=1}\nu_x^{\rm on}\prod\limits_{x:\,\xi'_x=0}\nu_x^{\rm off}}=\frac{\nu_{\bar x}^{\rm off}}{\nu_{\bar x}^{\rm on}}=
\frac{\Lambda_{(\xi', W'),(\xi, W)}}{\Lambda_{(\xi, W),(\xi', W')}}$$
Case (ii) can be analogously verified.

Consider now case (iii). We have that:
$$\frac{ \rho_{(\xi, W)}}{ \rho_{(\xi', W')}}=\frac{ W'_{\bar x\bar y'}  }{W_{\bar x\bar y}   }e^{\gamma (\Psi(W)-\Psi(W')}$$
On the other hand, using relation (\ref{transition2}) and the fact that $Z^{\bar x}_{\gamma}(\xi, W)=Z^{\bar x}_{\gamma}(\xi', W')$, we obtain
$$\frac{\Lambda_{(\xi', W'),(\xi, W)}}{\Lambda_{(\xi, W),(\xi', W')}}=\frac{W'_{\bar x\bar y'}}{W_{\bar x\bar y}}e^{\gamma ( f_{\bar x\bar y}(W)- f_{\bar x\bar y'}(W'))}$$
We can now conclude using relation (\ref{potential relation}). \qed
\end{proof}

We now show that under a slight stronger assumption than (\ref{cond allocation exists}), namely,
\beq\label{cond W ergodic}
\sum\limits_{x\in A}\alpha_x< \sum\limits_{y\in N(A)}\beta_y\quad\forall A\subseteq \mc X\,,
\eeq
the process $(\xi(t), W(t))$ restricted to $\{0,1\}^\mc X\times \mc W$ is ergodic. Denote by $\mc L$ the graph $\mc L_p$ restricted to the set $\mc W$. Notice that, as a consequence of time-reversibility, $\mc L$ is an undirected graph. Ergodicity is equivalent to proving that $\mc L$ is connected. We start with a lemma analogous to previous Lemma \ref{lemma equilibrium 1}.

\begin{lemma}\label{lemma equilibrium 2} Suppose $(\mc G ,\alpha, \beta)$  satisfies (\ref{cond W ergodic}) and let $W\in \mc W$. Then, for every $\bar y\in\mc X$, there exists a sequence (\ref{sequence})
satisfying the conditions (Sa), (Sb), and (Sc) as in Lemma \ref{lemma equilibrium 1}. 
\end{lemma}
\begin{proof} It is sufficient to follow the steps of to the proof of Lemma \ref{lemma equilibrium 1} noticing that in (\ref{estim}) the first equality is now a strict inequality, while the last strict inequality becomes an equality. \qed
\end{proof}

If $W,W'\in\mc W$ are connected through a path in $\mc L$, we write that $W\sim W'$. Introduce the following distance on $\mc W$: if $W^1, W^2\in\mc W$
$$\delta(W^1,W^2)=\sum\limits_{x,y}|W^1_{xy}-W^2_{xy}|$$
A pair $(W^1, W^2)\in\mc W$ is said to be minimal if 
$$\delta (W^1,W^2)\leq \delta (W^{1'}, W^{2'})\;\;\forall W^{1'}\sim W^1,\; \forall W^{2'}\sim W^2$$
Notice that $\mc L$ is connected if and only if for any minimal pair $(W^{1}, W^{2})$, it holds $W^{1}=W^{2}$. 

\begin{lemma}\label{lemma minimal1} Let $(W^1, W^2)$ be a minimal pair.
Suppose $y\in\mc X$ is such that $W^1_y<\beta_y$. Then, $W^1_{xy}=W^2_{xy}$ for all $x\in\mc X$.
\end{lemma}
\begin{proof}
Suppose by contradiction that $W^1_{xy}<W^2_{xy}$ for some $x\in\mc X$. Then, necessarily, there exists $y'\neq y$ such that $W^1_{xy'}>W^2_{xy'}$. Consider then $W^{1'}=W^1-e^{xy'}+e^{xy}$. Clearly, $\delta (W^{1'},W^2)<\delta(W^1,W^2)$ and this contradicts the minimality assumption. Thus $W^1_{xy}\geq W^2_{xy}$ for all $x\in\mc X$. This yields $W^2_y<\beta_y$. Exchanging the role of $W^1$ and $W^2$ we obtain the thesis. \qed
\end{proof}

\begin{proposition}\label{prop connected} If condition (\ref{cond W ergodic}) holds true, the graph $\mc L$ is connected.
\end{proposition}
\begin{proof}
Let $(W^1, W^2)$ be any minimal pair. We will prove that $W^1$ and $W^2$ are necessarily identical. Consider any resource $y$. It follows from Lemma \ref{lemma equilibrium 2} that we can find a sequence $y=y_0,\,x_0,\,y_1\cdots ,\,y_{t-1},\,x_{t-1},\,y_t$ satisfying the same (Sa), (Sb), and (Sc) with respect to the state allocation $W^1$. Among all the possible sequences, choose one with $t$ minimal for given $y$. We will prove by induction on $t$ that $W^1_{xy}=W^2_{xy}$ for all $x\in\mc X$.

If $t=0$, it means that $W^1_{y}<\beta_{y}$. It then follows from Lemma \ref{lemma minimal1} that $W^1_{xy}=W^2_{xy}$ for all $x\in\mc X$. Suppose now that the claim has been proven for all minimal pairs $(W^1,W^2)$ and any $y\in\mc X$ for which $t< \bar t$ (w.r. to $W^1$) and assume that $y=y_0,\,x_0,\,y_1\cdots ,\,y_{\bar t-1},\,x_{\bar t-1},\,y_{\bar t}$ satisfyies the properties (Sa), (Sb), and (Sc) with respect to $W^1$. Notice that the unit $x_{\bar t-1}$ can move a data atom from resource $y_{\bar t-1}$ into resource $y_{\bar t}$ under the state allocation $W^1$ and obtain $W^{1'}=W^1-e^{x_{\bar t-1}y_{\bar t-1}}+e^{x_{\bar t-1}y_{\bar t}}$. Consider now $W^2$ and notice that Lemma \ref{lemma minimal1} yields $W^2_{x_{\bar t-1}y_{\bar t}}=W^1_{x_{\bar t-1}y_{\bar t}}<\beta_{y_{\bar t}}$. Define
$$W^{2'}=\left\{\begin{array}{ll}W^2\;&{\rm if}\, W^2_{x_{\bar t-1}y_{\bar t -1}}=0\\
W^2-e^{x_{\bar t-1}y_{\bar t-1}}+e^{x_{\bar t-1}y_{\bar t}}\;&{\rm if}\, W^2_{x_{\bar t-1}y_{\bar t -1}}>0\end{array}\right.$$
Clearly, $\delta(W^{1'},W^{2'})\leq \delta(W^{1},W^{2})$ and this implies that also $W^{1'},W^{2'}$ is a minimal pair. Notice that $y=y_0,\,x_0,\,y_1\cdots ,\,y_{\bar t-1}$ satisfies (Sa), (Sb), and (Sc) with respect to $W^{1'}$. Therefore, by the induction hypotheses, it follows that $W^{1'}_{xy}=W^{2'}_{xy}$ for all $x\in\mc X$. Since $W^1_{xy}=W^{1'}_{xy}$ and $W^2_{xy}=W^{2'}_{xy}$, result follows immediately.
\qed
\end{proof}

Propositions \ref{prop time rev} and \ref{prop connected} immediately yield the following final result.
\begin{corollary}\label{cor main 2} Suppose that $(\mc G ,\alpha, \beta)$  satisfies (\ref{cond W ergodic}). Then $(\xi(t), W(t))$, restricted to $\{0,1\}^{\mc X}\times\mc W$, is an ergodic time reversible Markov process whose unique invariant probability measure is given by
$$\mu_{\gamma}(\xi, W)=Z_\gamma^{-1}\left[\prod\limits_{x:\,\xi_x=1}\nu_x^{\rm on}\prod\limits_{x:\,\xi_x=0}\nu_x^{\rm off} \right]{\alpha\choose W}e^{\gamma \Psi(W)}$$
where $Z_{\gamma}$ is the normalizing constant.
\end{corollary}

\medskip
{\bf Remark:} It follows from previous result that the process $W(t)$ converges in law to the probability distribution
$$\tilde\mu_{\gamma}(W):=\tilde Z_\gamma^{-1}{\alpha\choose W}e^{\gamma \Psi(W)}$$
Notice that when $\gamma\to +\infty$, the probability $\tilde\mu_{\gamma}$ converges to a probability concentrated on the set $\argmax_{W\in\mc W} \Psi(W)$ of state allocations maximizing the potential, thus on particular  Nash equilibria. Thus, if $\gamma$ is small, the distribution of the process $W(t)$ for $t$ sufficiently large will be close to a Nash equilibrium.

\medskip
{\bf Remark:} Condition (\ref{cond W ergodic}) is necessary for ergodicity. Notice indeed that in the case when $\mc G$ is complete and $\alpha_x=\beta_x=a$ for all $x\in\mc X$, under every allocation $W$ such that $W^y=a$ for every $y$, all resources will be saturated and, consequently, no distribution move will be allowed in $W$. Such allocations $W$ are thus all sinks in the graph $\mc L$ that is therefore not connected.

\section{Simulation}\label{sec simulations}

In this section, we present a number of numerical simulations that validate our theoretical results and show the behavior of the proposed algorithm. We analyze the structure of the asymptotic Nash configuration reached and the effect of the variations of the parameters in the utility function. The algorithm is tested for its convergence properties and evaluated through some performance indices measuring the computational and topological complexity of the solution found.

Below we report the basic assumptions of our simulations.

\begin{itemize}
\item Most of our simulations will be for $n=10$ or $n=50$ units while larger populations will be considered to show the scalability of the algorithm. 

\item We consider two possible underlying networks, a complete graph and a regular graph with degree $10$. 

\item We consider the case of homogeneous populations where all units have the same $\alpha_x$, $\beta_x$ and $\lambda_x$, as well heterogeneous cases with two subpopulations $\mc X_1$ and $\mc X_2$ sharing different parameters.

\item In the utility function we set the congestion parameter to be $k_c=1$ while we consider different values for the aggregation parameter $k_a$.

\item Most of our simulations will assume that units are always on. An example where units may be in off mode is however presented at the end.

\item In the implementation of the algorithm, time is supposed to be discrete, assuming that at each instant of time a unit $x$ wakes up with a probability proportional to $\alpha_x$. The allocation and distribution moves are chosen following the probabilities:
$$P_{\rm all}(W(t),x)=\left\{\begin{array}{ll} 1\;&{\rm if}\hspace{.2cm} W(t)^x<\alpha_x\\ 0\; &{\rm otherwise}\end{array}\right.,\quad P_{\rm dis}(W(t),x)=1-P_{\rm all}(W(t),x)$$

\item The parameter $\gamma$ in the Gibbs distribution is chosen to be time-varying with the law:
$$\gamma(t)=\gamma(t-1)+\frac{1}{\lambda_{max}*100}$$
where $\lambda_{max}$ is the maximum reliability of the resources. 

\item The typical time horizon we will consider in the algorithm will be $T=10\sum_x\alpha_x$: in this way every unit $x$ will activate, on the average, a number of times equal to $10$ times the number of data atoms it needs to allocate. As we will see, this time range is in general sufficient to reach, or at least get close, to a Nash equilibrium. Greater time ranges will be considered in some cases.

\end{itemize}

For all examples, the performance of the algorithm will be analyzed considering the following parameters computed, in a Montecarlo style, by averaging over 10 running of the algorithm.
\begin{itemize}
\item {\bf Distance from optimum:} $\psi=\frac{\Psi(W^{T)}}{\Psi_{opt}}$ where $\Psi_{opt}$ is the maximum of the potential. This allows to measure how the final allocation state $W(T)$ is far from a Nash configuration maximizing the potential $\Psi$. We will use the exact value of  $\Psi_{opt}$ in the homogeneous cases treated in Examples \ref{ex: homogeneous1} and \ref{ex: homogeneous2} for which the optimum Nash is explicitly known. For the heterogeneous cases we will instead approximate it using standard Matlab numerical algorithm.

In some cases we will also present a plot of the function $\psi(t)=\frac{\Psi(W(t))}{\Psi_{opt}}$ in this case for a single running of the algorithm.

\item {\bf Allocation complexity:} the parameter 
$$\nu_{moves}=\frac{1}{n}\sum\limits_{x\in\mc X}\frac{m_x}{\alpha_x}$$
where $m_x$ is the total number of moves (allocation and distribution) made by unit $x$. $\nu_{moves}$ measures the number of allocation and distribution moves per piece of data. Since moving data from a resource to another can be expensive, it is an important parameter to control.

\item {\bf Average satisfaction:} $$\bar\Lambda:=\frac{1}{n}\sum_{x\in\mc X}\sum_{y\in N_x}\frac{W^{T}_{xy}}{\alpha_x}\lambda_y$$
This is meaningful when there are resources with different values for $\lambda_y$. 
If we interpret $\lambda_y$ as a measure of the quality of a resource $y$ (reflecting possible different features as responsiveness, security, etc), $\bar\Lambda$ 
measures the efficiency of the algorithm to choose the best resources where to allocate.
\item{\bf Topological complexity:} We consider the subgraph constituted by the edges $(x,y)$ such that $W^{T}_{xy}>0$ and we compute the average out-degree and average in-degree for the two subpopulations in the heterogeneous case
$$d^+:=\frac{1}{n}\sum_{x\in\mc X}\sum_{y\in\mc X}\1_{\{W^T_{xy}>0\}},\; d^-_i:=\frac{2}{n}\sum_{x\in\mc X}\sum_{y\in\mc X_i}\1_{\{W^T_{xy}>0\}},\; i=1,2$$
$d^+$ measures the way units have split their data among the available neighbors, while $d^-_i$ the average amount of units using a resource of a certain type.
\item {\bf Average congestion:} 
For the heterogeneous case, the following index measure the congestion of the two classes of resources:
 $$\bar C^i:=\frac{1}{n\beta }\sum_{y\in\mc X_i}W^{T}_y$$
\end{itemize}

\subsection{Homogeneous case}
In this section we consider the case when all units have equal parameters: $\alpha_x=\alpha$, $\beta_x=\beta$ and $\lambda_x=\lambda$ for every $x\in\mc X$. In this case Nash equilibria were studied in Examples \ref{ex: homogeneous1} and \ref{ex: homogeneous2}. In particular, we know that, for any value of the aggregation parameter $k_a$, the maximum of the potential is given by the Nash equilibria where each unit uses just one resource: $W^\sigma_{xy}=\alpha\1_{\sigma(x)=y}$ where $\sigma :\mc X\to\mc X$ is any permutation without fixed points.

\begin{example}\label{ex: completo}

Assume that $n=10$ and that the underlying graph is complete. Parameters are set to the values: $\alpha=27$, $\beta=30$, and $\lambda=3$.  We consider the following values for the aggregation term $k_a$: $0, 0.003, 0.1$. Notice that only case when $k_a>1/\beta$ is for $k_a=0.1$: in this case we expect to find final state allocation where the aggregation feature prevails according to the considerations in Example \ref{ex: homogeneous2}.

For all simulations we compute the indices $\psi$, $\nu_{moves}$ and $d^+$. Here are the values we have obtained:
\begin{center}
\begin{table}[h]
\centering
\scriptsize
\begin{tabular}{|l|c|c|c|c|c|}
\hline
&$k_a=0$  & $k_a=0.003$ & $k_a=0.1$\\
\hline
$\nu_{moves}$&  8.7863  & 6.2907 &  3.5359\\[1pt]
$d^+$ &   7.4100   &    1& 1 \\[1pt]
$\psi$ &  0.9994   &  1  & 1 \\
\hline
\end{tabular}
\caption{Values for $k_a=0$, $k_a=0.003$, $k_a=0.1$ in the complete case}
\end{table}
\end{center}

\vspace{-12pt}

Notice how the number of moves per atom is decreasing with $k_a$: this feature is due to the fact that the stronger is the aggregation term, the more the initial allocation will tend to immediately create a solution with aggregation features already close to a Nash equilibrium. 

Notice how for $k_a=0$ the optimal is in general not reached, while it is reached in the case $k_a=0.003$ and  $k_a=0.1$. In order to better understand the convergence behavior, below we present, for a single running of the algorithm, the plot of $\psi(t)$:

\begin{figure}[htb]
\centering
\begin{minipage}{.28\textwidth}
\centering
\includegraphics[scale=.18]{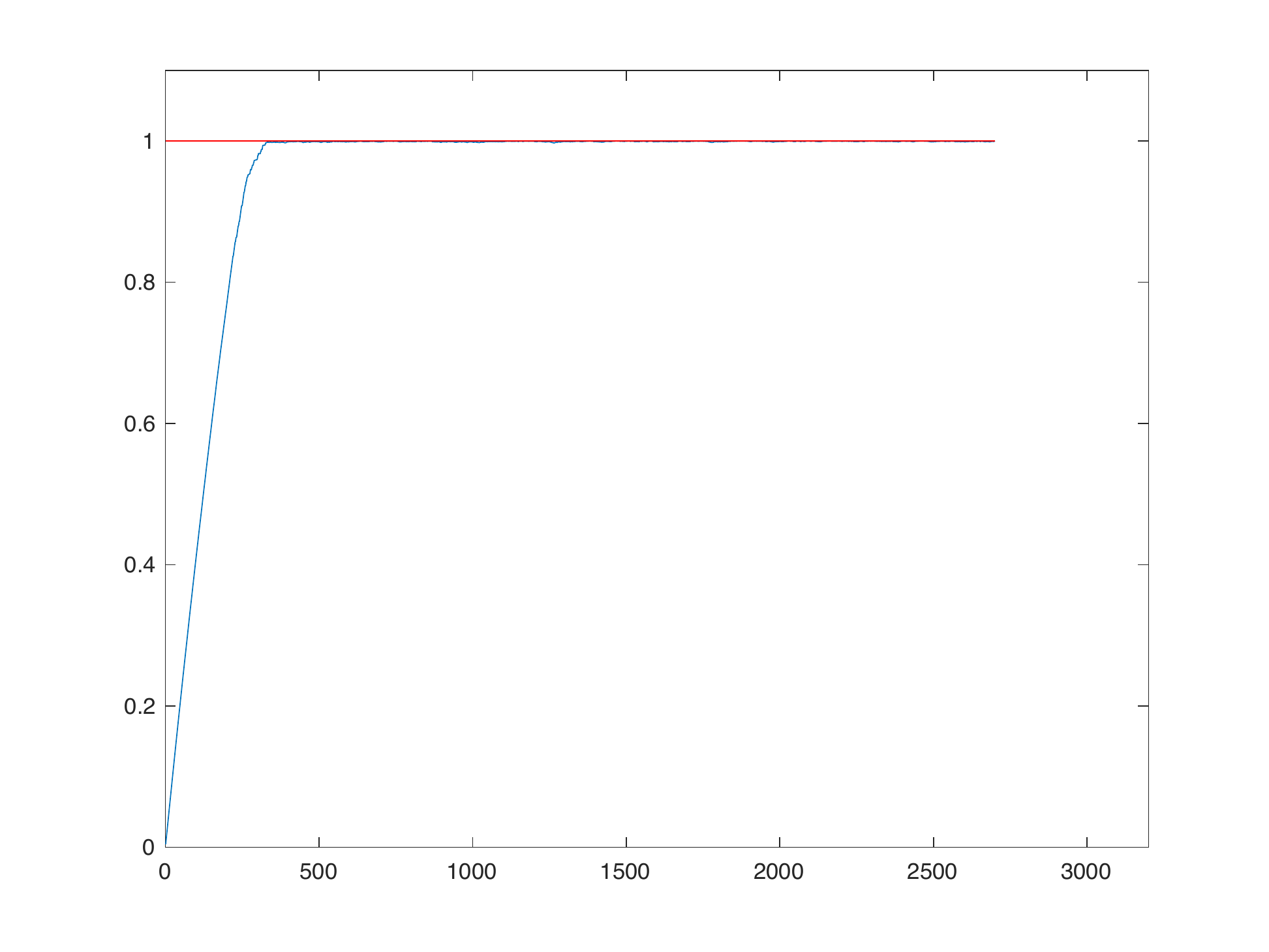}
\end{minipage}%
\begin{minipage}{.28\textwidth}
\centering
\includegraphics[scale=.18]{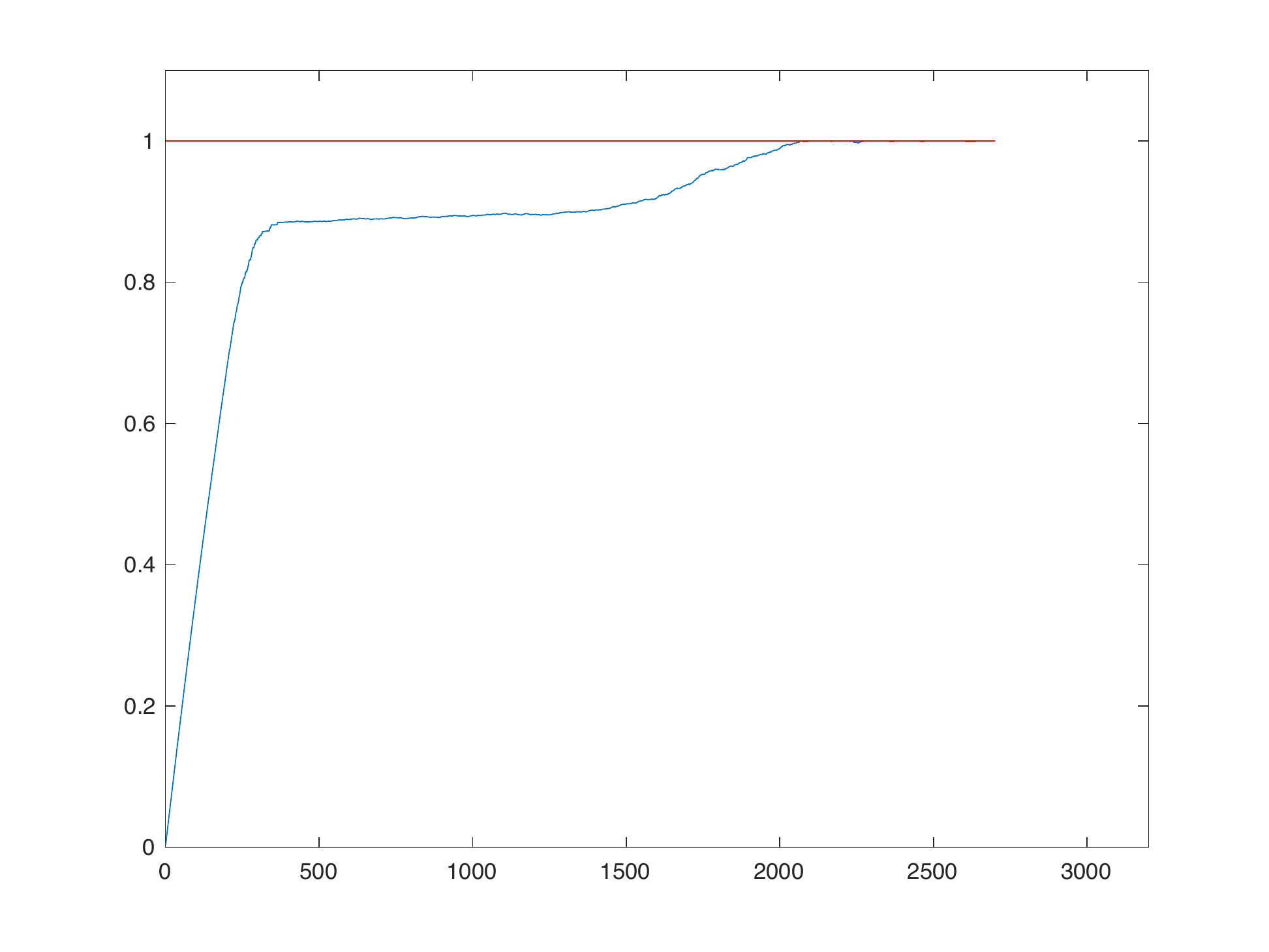}
\end{minipage}%
\begin{minipage}{.28\textwidth}
\centering
\includegraphics[scale=.18]{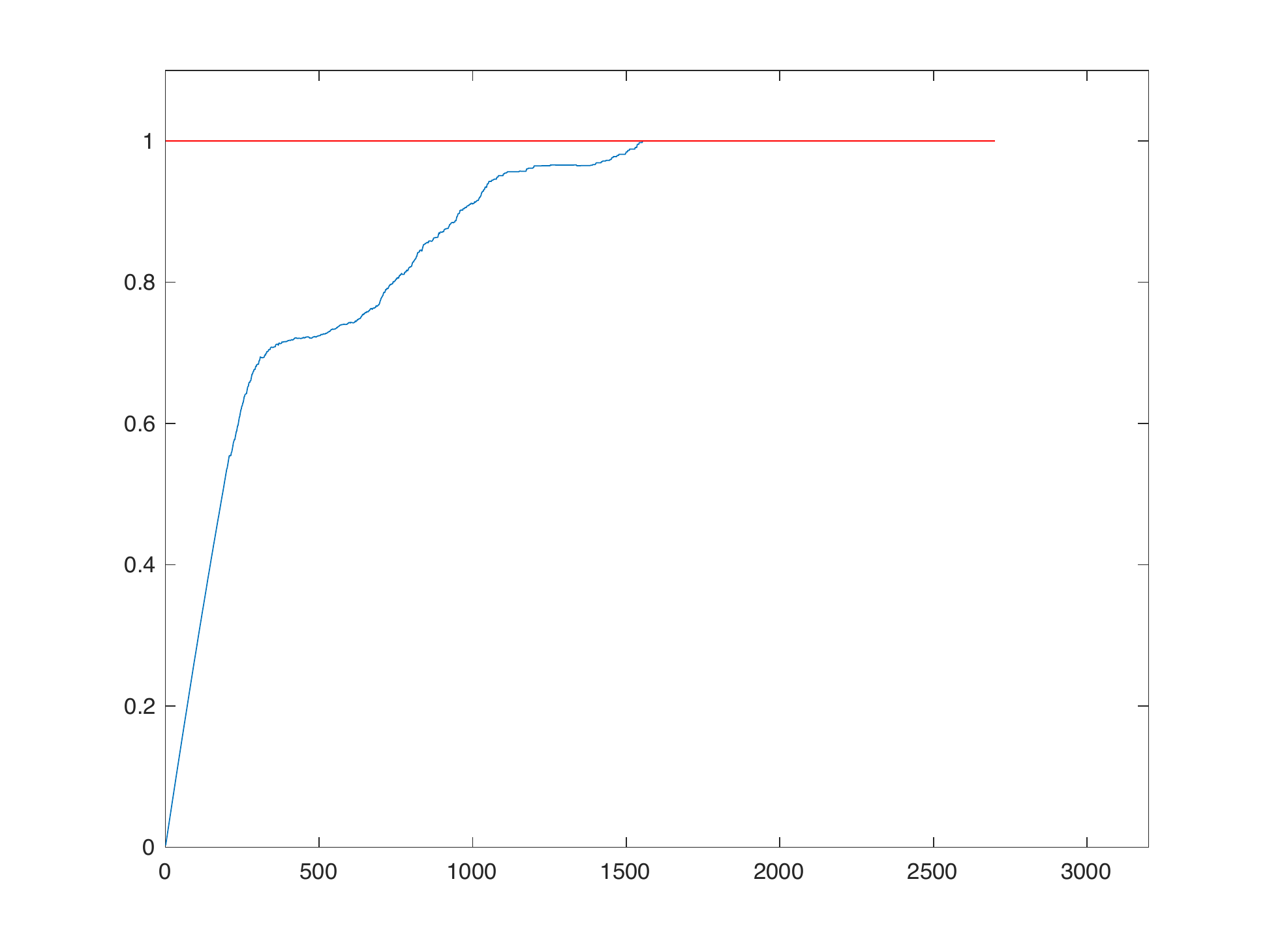}
\end{minipage}
\caption{Plot of $\psi(t)$ for $k_a=0$, $k_a=0.003$, $k_a=0.1$}
\end{figure}


Notice how, in the case $k_a=0$, the convergence is indeed fast in the initial part but then it slows down and the dynamics remains for a long time fluctuating in matrices close to a Nash. A typical running of the algorithm yields a final configuration $W(T)$ of type
$$W(T)=\left(\begin{matrix}
 0    & 4   &  2   &  2  &   1   &  1   &  5   &  8   &  2  &   2 \\
    4   &  0  &   8  &   3    & 3 &    2     &4     &1   &  1   &  1\\
    3   &  2 &    0 &    2  &   4   &  1  &   1   &  5     &4    & 5\\
    1  &   5  &   5  &   0  &   1   &  0  &   0  &  10    & 3   &  2\\
    6  &   0  &   3   &  1  &   0   &  5  &   3    & 2   &  4  &   3\\
    2  &   1  &   2   &  1 &    2   &  0   &  6&     0    & 7 &    6\\
    3   &  4 &    0   &  5  &   7    & 3   &  0  &   1  &   1 &    3\\
    2   &  5 &    2   &  4 &    4  &   1   &  3    & 0&     2  &   4\\
    3 &    5 &    4   &  2  &   2    & 4   &  1   &  3     &0   &  3\\
    5&     0 &    3   &  6&     3   &  5   &  3     &0   &  2  &   0 \\
    
\end{matrix}\right)$$
\normalsize
with congestion vector 
$C=\1^*W(T)=(\begin{matrix}29  &  26   & 29  &  26  &  27  &  22    &26 &   30 &   26  &  29 \end{matrix})$
In this case, in order to achieve the optimum with all resources equally congested, the algorithm must run for much longer time, order of $200*\sum_x\alpha_x$. This is a peculiarity of the case without aggregation term and such phenomena never shows up otherwise. A faster convergence however can be obtained in this case by a different tuning of the noise function. Using 
$$\gamma(t)=\gamma(t-1)+\frac{1}{\lambda_{max}*10}$$ we typically reach an optimum at $T=30*\sum_x\alpha_x$.

In the case when $k_a=0.003$ or $k_a=0.1$ convergence is slower in the initial part, however, the typical state allocation $W(T)$ constructed by the algorithm will instead be a matrix of type $W^{\sigma}$, namely an optimal Nash equilibrium.

\end{example}

The aim of the following example is to prove that the different non-optimal Nash equilibria presented in Example \ref{ex: homogeneous2} can attract the dynamics of the algorithm.

\begin{example}

Assume that $n=10$ and that the underlying graph is complete. Parameters are set to the values: $\alpha=27$, $\beta=60$  (so that $\beta>2\alpha$), $\lambda=3$, and $k_a=0.1>1/\beta$. The time horizon is fixed to $T=20*\sum_x\alpha_x$. Here are the values of $\psi$, $\nu_{moves}$ and $d^+$ obtained in this case:
\begin{center}
\begin{table}[h]
\centering
\scriptsize
\begin{tabular}{|l|c|c|c|}
\hline
$\nu_{moves}$&3.2781 \\[1pt]
$d^+=d^-$ &   1      \\[1pt]
$\psi$ &   0.9805    \\
\hline
\end{tabular}
\caption{Values for $\beta>2\alpha$}
\end{table}
\end{center}

\vspace{-12pt}
Notice how the optimal Nash (e.g. any aggregated matrix $W^{\sigma}$) is not reached in general. If we analyze single runnings of the algorithm, we can verify that the algorithm reach different possible Nash equilibria: besides the optimal $W^{\sigma}$, we observe also final configurations where one or more resources are used by two units. Here is a sample of final configurations observed

  $$W=\left(\begin{matrix}   0 &  0 &  0 &  0 &  0 & 27 &  0 &  0 &  0 &  0\\
      0 &  0 &  0 & 27 &  0 &  0 &  0 &  0 &  0 &  0\\
      0 &  0 &  0 &  0 &  0 &  0 &  0 &  0 &  0 & 27\\
     27 &  0 &  0 &  0 &  0 &  0 &  0 &  0 &  0 &  0\\
      0 &  0 &  0 &  0 &  0 &  0 &  0 & 27 &  0 &  0\\
      0 &  0 &  0 &  0 &  0 &  0 &  0 & 27 &  0 &  0\\
      0 &  0 &  0 &  0 & 27 &  0 &  0 &  0 &  0 &  0\\
      0 &  0 &  0 &  0 &  0 &  0 &  0 &  0 & 27 &  0\\
      0 &  0 & 27 &  0 &  0 &  0 &  0 &  0 &  0 &  0\\
      0 &  0 &  0 &  0 &  0 &  0 & 27 &  0 &  0 &  0\\
\end{matrix}\right)$$

\normalsize

Notice how such Nash equilibria where indeed correctly forecasted in Example \ref{ex: homogeneous2}.

\end{example}
We conclude the homogeneous case with an example showing the good scalability properties of the algorithm.
\begin{example}

We consider the case when $n=50,100,1000$ and that the underlying graph is a random regular graph with degree 10 (the use of the complete graph for large number of units becomes heavy on the computational side and, in any case, is also of little practical interest). Parameters are set to the values: $\alpha=27$, $\beta=30$, $\lambda=3$, $k_a=0.03$.  
Here are the values obtained for $\psi$, $\nu_{moves}$ and $d^+$:

\begin{center}
\begin{table}[h]
\centering
\scriptsize
\begin{tabular}{|l|c|c|c|}
\hline
& $n=50$ & $n=100$ & $n=1000$ \\
\hline
$\nu_{moves}$& 2.6876 &  2.7064 & 1.4592\\[1pt]
$d^+=d^-$ &   1  &1.0870   &  1.1000\\[1pt]
$\psi$ &   1 &    0.9978  & 0.9973\\
\hline
\end{tabular}
\caption{Scalability of the Algorithm}
\end{table}
\end{center}   
\end{example}

\subsection{Non-homogeneus case}
In this section we consider the case when all units have the same $\alpha_x=\alpha$ and $\beta_x=\beta$ but, for what the reliability is concerned, the population is split into two parts $\mc X_1$ and $\mc X_2$ and 
$\lambda_x=\lambda^i$ if $x\in\mc X_i$ with $\lambda^1\neq \lambda^2$.
While we do not have in this case any explicit computation of the optimal Nash equilibrium,  it can easily been shown that the aggregate solution $W^{\sigma}$ is a Nash Equilibrium (possibly not a maximum of the potential) if and only if  $k_a\geq \frac{\lambda^2-\lambda^1-1/\beta}{\alpha-1}$.

\begin{example}\label{ex: etero on}
Assume that $n=50$, $|\mc X_1|=|\mc X_2|=25$, and that the underlying graph is either complete or random regular of degree $10$. Parameters are set to the values: $\alpha=45$, $\beta=50$, $\lambda^1=0.5$ $\lambda^2=0.8$.  We consider different values for the aggregation term $k_a$.
For all simulations we compute the parameters $\psi$, $\nu_{moves}$, $d^+$, $d^-$, $\bar\Lambda$, and $\bar C^i$ averaged over 10 running of the algorithm. Here are the values we have obtained in the two cases of a complete graph and of a random regular graph of degree $10$.

\begin{table}[htbp]
        \centering
        \begin{minipage}[c]{.45\textwidth}\tiny
         \centering
\begin{tabular}{|l|c|c|c|c|c|}
\hline
 &$k_a=0.001$  & $k_a=0.005$ & $k_a=0.1$ \\
\hline
$\nu_{moves}$&  4.7015  & 4.8150 &  3.2110 \\ [1pt]
$\bar \Lambda$ &   0.6667 &   0.6593 & 0.6502 \\ [1pt]
$\bar C_1$ &    0.8000     &  0.8439  &0.9000  \\ [1pt]
$\bar C_2$ &   1.0000   &     0.9561 &0.9000 \\ [1pt]
$d^+$ &   21.1180  &    2.1680& 1 \\ [1pt]
$d^-_1$ &  13.5440   &  1.0920 & 1\\ [1pt]
$d^-_2$ &  28.6920  &   3.2440&1\\ [1pt]
$\psi$ &  0.9998   &  1  & 1 \\
\hline
\end{tabular}
          \caption{Complete case}
        \end{minipage}%
        \hspace{10mm}%
        \begin{minipage}[c]{.45\textwidth}\tiny
          \centering
\begin{tabular}{|l|c|c|c|}
\hline
 &$k_a=0.001$  & $k_a=0.005$ & $k_a=0.1$\\
\hline
$\nu_{moves}$& 4.1580   &  4.5907 & 2.5242 \\[1pt]
$\bar \Lambda$ &  0.6667  &  0.6595  & 0.6511 \\[1pt]
$\bar C_1$ &   0.8000      & 0.8430  & 0.8933  \\[1pt]
$\bar C_2$ & 1.0000     &  0.9570    &  0.9067\\[1pt]
$d^+$ &  6.3160   &   1.7000  & 1.0680 \\[1pt]
$d^-_1$ &   2.8720  &  1.0480 & 1.0240\\[1pt]
$d^-_2$ & 9.7600   &  2.3520  &1.1120\\[1pt]
$\psi$ &   0.9998   &  0.9999   & 0.9986\\
\hline
\end{tabular}
          \caption{Regular case}
        \end{minipage}

      \end{table}
\normalsize

These results, on one hand, show how the speed and computational performance of the algorithm in these heterogeneous cases are similar to the homogeneous case: at the horizon time $T$ the system is always very close to (or is exactly at) the optimal solution and the number of moves per atom is similar. On the other hand, other interesting features emerge from an analysis of the other parameters. The congestion parameters show how, for very small aggregation term $k_a=0.001$, the most reliable resources in $\mc X_2$ are fully used, while for $k_a=0.1$ resources in $\mc X_1$ and $\mc X_2$ are almost equally congested. Indeed in this last case, at least when the graph is complete, the aggregated solutions $W^{\sigma}$ are optimal and they are indeed reached by the algorithm. Notice that the average degrees are already very small at the intermediate level $k_a=0.005$: in this case the solution found by the algorithm is close to one where each unit uses just one resource. Clearly, the tuning of the parameters (in particular, the aggregation term $k_a$) will be important in real world applications: it will allow, depending on the specific needs of the working case, to trade-off reliability of the mostly used resources versus aggregation.
Finally, notice how the fact of working with a preassigned topology does not change very much the behavior of the algorithm.

\end{example}

We conclude this section with an example where we also incorporate the feature that resources may be on or off. In our discrete time model we assume that every time each unit $x$ is on with probability $p_x$ (independently from the others). 
Specifically, we assume that $\mc X$ is split into two subfamilies $\mc X_1$ and $\mc X_2$ such that $\lambda_x=\lambda^i$ and $p_x=p^i$ if $x\in\mc X_i$. Moreover, we assume that $\lambda^i=m p^i$ for $i=1,2$ and some proportionality constant $m$. This makes sense: the reliability is assumed to be proportional to the probability of finding the resource in on state.

\begin{example}

Assume that $n=50$, $|\mc X_1|=|\mc X_2|=25$, and that the underlying graph is complete. Parameters are set to the values: $\alpha=45$, $\beta=50$, $p^1=0.5$, $p^2=0.8$, $m=5$, $k_a=0,1$.  Here, we fix the time horizon $T=50*\sum_x\alpha_x$. 

We compute the parameters $\psi$, $\nu_{moves}$, $d^+$, $d^-$, $\bar\Lambda$, and $\bar C^i$ averaged over 10 running of the algorithm.

\vspace{-12pt}


\begin{center}
\begin{table}[h]
\centering
\scriptsize
\begin{tabular}{|l|c|c|c|c|c|}
\hline
$\nu_{moves}$&  3.1080    \\[1pt]
$\bar \Lambda$ &   0.6620\\[1pt]
$\bar C_1$ &  0.8278      \\[1pt]
$\bar C_2$ &0.9722\\[1pt]
$d^+$ &   1.4860  \\[1pt]
$d^-_1$ & 0.9320   \\[1pt]
$d^-_2$ & 2.0400   \\[1pt]
$\psi$ &     0.9484  \\
\hline
\end{tabular}
\caption{Values considering on and off resources}
\end{table}
\end{center}
\normalsize
\vspace{-12pt}
As shown in the table, even if the time horizon is higher than in previous examples, the average number of moves per each atom is of similar order due to the fact that when units are in off mode they can neither allocate or accept allocation. For the rest, the other parameters show that the algorithm, though more slowly, has a behavior similar to the case when units are always on.

\end{example}

\section{Conclusions}
In this paper, we have analyzed an allocation algorithm that allows a network of units (computers)  to use the memory space available in their neighbors to store a back up of their data. The interest for such an algorithm comes from the goal of creating a peer-to-peer decentralized cloud storage model to be thought as an alternative to classical centralized cloud storage services.

The proposed algorithm is based on evolutionary game theory. The main theoretical contribution has been to prove convergence and to give en efficient characterization of the asymptotic in terms of an explicit invariant probability distribution picked on optimal Nash equilibria. We have also carried on a number of simulations showing the good scalability properties of the algorithm, its reduced complexity and illustrating the nice features of 
the final allocation state.

A number of important issues remain open and are left for future research. Among these: (i) find explicit bounds on the convergence time of the algorithm; (ii) study the structure of the Nash equilibria for general topologies;
(iii) extend to the case when multiple back up copies are needed to be stored with the further security constraint of using different resources for the different copies.

\section*{Acknowledgment}
We acknowledge that this work has been done while Barbara Franci was a PhD Student sponsored by a TIM/Telecom Italia grant.

\end{document}